\documentclass[10pt]{amsart}

\RequirePackage{fix-cm}
\usepackage[utf8]{inputenc} 

\usepackage{amsmath}
\usepackage{amssymb}
\usepackage{amsfonts, dsfont}
\usepackage{mathrsfs}
\usepackage{mathtools}
\usepackage{pstricks}
\usepackage{enumitem}
\newtheorem{theorem}{Theorem}[section]
\newtheorem{lemma}[theorem]{Lemma}
\newtheorem{corollary}[theorem]{Corollary}
\newtheorem{proposition}[theorem]{Proposition}
\newtheorem{remark}[theorem]{Remark}
\newtheorem{definition}[theorem]{Definition}

\theoremstyle{plain}

\numberwithin{equation}{section}

\DeclareMathOperator*{\argmin}{argmin}

\newcommand{\numberset}{\mathbb}
\newcommand{\R}{\numberset{R}}
\newcommand{\A}{\mathcal{A}}

\newcommand{\D}[1]{\mbox{\rm #1}} 
\newcommand{\dd}{\D{d}}

\def\eps{\varepsilon}
\def\a{\alpha}
\def\d{\delta}

\def\de{\partial}
\def\P{\mathcal{P}}

\def\M{\mathfrak{M}}
\def\A{\mathcal{A }}
\def\B{\mathcal{B }}

\usepackage{xcolor}

\begin{document}
\title[MFG with non-monotone interactions]{Long-time behavior of deterministic Mean Field Games with non-monotone interactions} 
\thanks{M.B. is member of the Gruppo Nazionale per l'Analisi Matematica, la Probabilit\`a e le loro Applicazioni (GNAMPA) of the Istituto Nazionale di Alta Matematica (INdAM). He also  participates in the King Abdullah University of Science and Technology (KAUST) project CRG2021-4674 ``Mean-Field Games: models, theory, and computational aspects'', 
and in the project funded by the EuropeanUnion-NextGenerationEU under the National Recovery and Resilience Plan (NRRP), Mission 4 Component 2 Investment 1.1 - Call PRIN 2022 No. 104 of February 2, 2022 of Italian Ministry of University and Research; Project 2022W58BJ5 (subject area: PE - Physical Sciences and Engineering)  ``PDEs and optimal control methods in mean field games, population dynamics and multi-agent models".\\
\indent H.K. was funded by the Deutsche Forschungsgemeinschaft (DFG, German Research Foundation) – Projektnummer 320021702/GRK2326 – Energy, Entropy, and Dissipative Dynamics (EDDy).}

\author{Martino Bardi}
\address {Martino Bardi \newline 
{Department of Mathematics “T. Levi-Civita”, 
University of Padova, via Trieste, 63}, 
{I-35121 Padova, Italy}
}
\email{\texttt{bardi@math.unipd.it}}

\author{Hicham Kouhkouh}
\address{Hicham Kouhkouh \newline 
{RWTH Aachen University, Institut f\"ur Mathematik,  
RTG Energy, Entropy, and Dissipative Dynamics, 
Templergraben 55 (111810),}
{52062, Aachen, Germany} \newline 
\texttt{Current address: } 
Department of Mathematics and Scientific Computing, 
NAWI, University of Graz, 
8010, Graz, Austria 
}

\email{\texttt{kouhkouh@eddy.rwth-aachen.de}, \quad \texttt{hicham.kouhkouh@uni-graz.at}}

\date{\today}
\begin{abstract}
We consider deterministic Mean Field Games (MFG) in all Euclidean space with a cost functional continuous with respect to the distribution of the agents and attaining its minima in a compact set. We first show that the static MFG with such a cost has an equilibrium, and we build from it a solution of the ergodic MFG system of 1st order PDEs with the same cost. Next we address the long-time limit of the solutions to finite horizon MFG with cost functional satisfying various additional assumptions, but not the classical Lasry-Lions monotonicity condition. Instead we assume that the cost has the same set of minima for  all  measures describing the population. We prove the convergence of the distribution of the agents and of the value function to a solution of the ergodic MFG system as the horizon of the game tends to infinity, extending to this class of MFG some results of weak KAM theory.
\end{abstract}



\maketitle

\section{Introduction}
We consider deterministic Mean Field Games (MFG) and the corresponding systems of first-order PDEs in two cases: problems with finite horizon $T$, whose system of PDEs is
\begin{equation}
\label{evMFG_0}
\left\{
\begin{aligned}
   & -\partial_{t}u^T + \frac{1}{2}|Du^T(x,t)|^{2} = F(x, m^T(t) ), & \quad \text{in } \mathds{R}^n \times (0, T) ,\\   
    & \partial_{t}m^T  -\text{div}(m^T \nabla u^T) = 0, & \quad \text{in }\mathds{R}^{n}\times (0, T)  ,\\ 
    & u^T(x,T) =\, 0, \quad m^T(0)= m_o,
\end{aligned}\right.
\end{equation}
in the unknowns $(u^T, m ^T)$, and problems with ergodic cost functional, leading to the system\vspace*{-1mm}
\begin{equation}
\label{erg_0}
\left\{\quad
\begin{aligned}
    c+ \frac 12 |\nabla v(x)|^2 &= F(x,m) \quad \text{in }\mathds{R}^{n} ,\\
    -\text{div}(m \nabla v) &= 0  \quad \text{in }\mathds{R}^{n} ,\quad \int_{\R^n} \dd m(x) = 1 ,
\end{aligned}\right.
\end{equation}
in the unknowns $c\in\R, v\in C(\R^n)$, and $m\in \P_1(\R^n)$. 
Our main assumptions are that the cost function $F : \R^n\times \P_1(\R^n) \to \R$ is continuous, where $\P_1(\R^n)$ denotes the probability measures with finite first moment endowed with the Kantorovich-Rubinstein distance $d_1$, and $F$ attains a minimum in $x$ for each fixed measure $m$ by virtue of the gap condition 
\begin{equation}
\label{Finfinity_0}
   \liminf\limits_{|x|\to\infty} F(x,m) > \min\limits_{x\in K_{\circ}} F(x,m)  
\end{equation}
for some compact set $K_{\circ}$. 

We first consider { static} Mean Field Games with cost $F$,  namely, 
\begin{equation}
    \label{supp pb_0}
    \text{find } \bar{m}\in\mathcal{P}_{1}(\mathds{R}^{n}) \text{ s.t.: } \; \text{supp}(\bar{m}) \subseteq \argmin\limits_{x\in\mathds{R}^{n}}\, F(x,\bar{m}) ,
\end{equation}
where  $\text{supp}({m})$ denotes the support of the measure $m$, and compare them with the{ ergodic} MFG associated to \eqref{erg_0}. We show that solutions to these two classes of games exist and are in one-to-one correspondence. Next we make stronger assumptions and study the long-time behavior of solutions of {the finite horizon} problem \eqref{evMFG_0}. The main results connect all possible limits, as the horizon $T$ goes to infinity, to solutions of the static and the ergodic MFG.

Throughout the paper we never assume the Lasry-Lions monotonicity condition on the cost, nor any other known condition implying uniqueness of the solution. This is the main difference of our results from most of the  papers treating the long-time limit in MFG. The main contributions in the deterministic case are by Cardaliaguet on the torus \cite{cardaliaguet2013long}, by Cannarsa, Cheng, Mendico and Wang in the whole space \cite{cannarsa2020long} and in bounded domains with state constraints \cite{cannarsa2021weak}, see also \cite{cardaliaguet2021ergodic}  for MFG depending on acceleration, all under a strengthened version of monotonicity. For stochastic MFG and the corresponding parabolic systems of PDEs we must mention the seminal papers by Cardaliaguet, Lasry, Lions and Porretta \cite{cardaliaguet2012long, cardaliaguet_l_l_p2013long}, the long-time limit of the master equation \cite{cardaliaguet2019long}, the turnpike property \cite{porretta2018turnpike}, and the case of non-compact state space \cite{arapostathis2017solutions}. They all rely on the Lasry-Lions monotonicity condition,  describing games where the cost of the agents increases in areas of higher density.  Cirant and Porretta \cite{cirant2021long} treat MFG with mild non-monotonicity and a white noise compensating for it: they introduce a parameter $\gamma$ measuring the loss of monotonicity and prove uniqueness  and asymptotic behavior of solutions for $\gamma$ small enough with respect to the diffusion and the sup norm of $m^T$ and $Du^T$.

Our motivation is to deal with problems where attraction among players can be preferable to repulsion, at least in some areas and for certain values of the density. A first example is the classical Lions' problem {\em ``Where do I put my towel on the beach?''} \cite{lions2008, carmona2018probabilistic} in the case that people do not necessarily prefer to stay far from the others and possibly there are areas of the beach that are more attractive than others. A second example is global unconstrained optimization of a function $f\in C(\R^n)$ and the  associated eikonal equation and generalized gradient descent studied in \cite{bardi2023eikonal}. The MFG considered in the present paper includes an extension of the time-optimal control problem of \cite{bardi2023eikonal} to the non-cooperative search of the minima of $f$ by a large population of interacting identical agents.

Let us outline the contents of the paper. In Section \ref{sec:static} we consider the static MFG with cost $F$ as in \eqref{supp pb_0}. It was derived by Lions \cite{lions2008} as the limit of Nash equilibria in $N$-person games as $N\to\infty$, see \cite{cardaliaguetnotes}. It is also related to Cournot-Nash equilibria for games with a continuum of players, see \cite{blanchet2014remarks} and the references therein. We give a simple existence result of an equilibrium $\bar m$ under the mere assumptions mentioned above.

Section  \ref{sec:ergodic} is about ergodic MFG with quadratic Lagrangian, whose associated system of PDEs is \eqref{erg_0}. The main result of the section is that for any solution $\bar m$ of the static game there exist $v$ such that $(c_{\bar m}, v, \bar m)$ solves the ergodic system \eqref{erg_0}, with $c_{\bar m}:=\min_xF(x, \bar m)$. Conversely, if $(c,v,m)$ solves \eqref{erg_0} then $c\leq c_m$ and $m$ solves the static game.

Section \ref{sec:approx} contains the main results of the paper, which concern the finite horizon MFG system \eqref{evMFG_0}, and the limit of its solutions as $T\to +\infty$. Here we make considerably stronger assumptions than \eqref{Finfinity_0}: $F$ is uniformly $C^2$ in $x$ and uniformly Lipschitz in $m$, as in the lecture notes \cite{cardaliaguetnotes}, and we adopt all assumptions of \cite{cannarsa2020long}, except the monotonicity of the cost. In particular, the support of the initial measure $m_o$ is compact and $\cap_{m \in  \P_1(\R^n)} \argmin F(\cdot,m) $ is nonempty.
We first look at $m^T$ and prove that for any $s\in (0, 1]$ the family $\{ m^{T}(sT) , T>1 \}$ is precompact and any weak-$*$ limit $\bar m(s)$ of a subsequence with $T_n\to\infty$ has the support in the set $\cup_{m \in  \P_1(\R^n)} \argmin F(\cdot,m)$. Then, under the additional assumption that for some $\A\subseteq\R^n$
\begin{equation}
\label{strongass}
\argmin_x F(x,m) = \A \quad \forall\, m \in  \P_1(\R^n) \,,
\end{equation}
we get that $\text{supp} (\bar m(s))\subseteq \A$, and therefore the limit $\bar m(s)$ solves the static game \eqref{supp pb_0} as well as the ergodic MFG. The condition \eqref{strongass} replaces the monotonicity condition used in \cite{cardaliaguet2013long} and \cite{cannarsa2020long} and is completely different: it says that the set $\A$ is the best place to stay for all agents, regardless of the population distribution. Our model problem is 
\[
F(x,m)=f(x) g\left(\int_{\R^n}k(x,y) \,\dd m(y)\right) \,, \quad f\geq 0\,,\, k\geq 0\,,\, g(r)\geq 1 \;\forall\, r\geq 0 \,,
\]
where $f, k,$ and $g$ are $C^2$ with bounded derivatives and $f$ vanishes at some point. Then \eqref{strongass}  is satisfied with $\A =\argmin f$.

Next we consider the convergence of $u^T$, under the further condition
\begin{equation}
\label{c*_0}  c_* = \min_x F(x,m)   \quad \forall m.
 \end{equation}
 As expected from ergodic control and weak KAM theory we get 
  \[
 \lim_{T\to +\infty}  \frac {u^{T}(x, t)}{T-t} = c_*\quad\text{locally uniformly in } \R^n\times [0,T[ \,.
 \]
We can also give the following estimate of the rate of convergence
\begin{equation}
\sup_{|x|\leq R ,\, 0\leq s \leq 1} \left | \frac {u^T(x, sT)}T - c_*(1-s)\right | \leq  \frac {C(R)}T \,. 
\end{equation}
Remarkably, this is better than the estimates in  \cite{cardaliaguet2013long} and \cite{cannarsa2020long}, where the right hand side is at best of the order $1/\sqrt T$.
 
Finally we aim at a further step, in the spirit of weak KAM theory, namely 
\begin{equation}
\label{lim_sing}
\lim_{T\to \infty}\left(u^T(x, sT) - c_*T(1-s)\right) = v(x,s) \,,
\end{equation}
 locally uniformly for $(x,s)\in \R^n\times (0,1]$, 
where $v(\cdot, s)$ is a solution of the first equation in the ergodic problem \eqref{erg_0} with $m=\bar m(s)$. In general we manage to prove only a weaker result. However, in the special case that $\A=\{x_*\}$ is a singleton, we prove that the limit \eqref{lim_sing} is true. Moreover $(c_*, v, \d_{x_*})$ solves the ergodic mean field game system \eqref{erg_0} and it is essentially the unique, or maximal, solution, in the sense that  the $m$ component of any solution must be  $\d_{x_*}$, the $c$ component must be $\leq c_*$, and  $v$ is   the unique  non-negative viscosity solution of the critical eikonal equation
 \begin{equation}
c_*+ \frac 12 |\nabla_x v|^2 = F(x, \d_{x_*}) \;\text{ in } \R^n \setminus\{x_*\} , \quad v(x_*)=0 \,. 
\end{equation}

The paper is restricted to the simplest case of quadratic Lagrangian and Hamiltonian, but many results can be extended to more general Hamiltonians, although at least coercive and smooth. Here we do not try do deal with larger generality because our focus is rather on the cost $F$ and its lack of monotonicity.

 We conclude with some additional references on MFGs without the repulsive monotonicity condition.  Deterministic variational games with attractive cost were studied by  Cesaroni and Cirant \cite{cesaroni2021brake}, who found periodic orbits in some cases.
Gomes, Mitake and Terai analysed the small discount limit \cite{gomes2020selection}. 

For stochastic potential games  the  long-time behavior  was studied  by Masoero and Cardaliaguet \cite{cardaliaguet2020weak, masoero2019long}.  Second order MFG systems with local cost of the form $F(x,m)=V(x) - Cm(x)^\a$, $C, \a>0$, which induces the aggregation  of the agents, were considered by Cirant  in the stationary case \cite{cirant2016stationary}, jointly with Tonon \cite{cirant2019time} and Ghilli  \cite{cirant2022existence} in the evolutive case. Cesaroni and Cirant  also studied concentration phenomena in the vanishing viscosity limit \cite{cesaroni2018concentration}. The recent papers \cite{bernardini2022mass, bernardini2023ergodic} study aggregative MFGs with nonlocal singular costs of Choquard type.

\section{Static  MFG}
\label{sec:static}

Consider $F : \R^n\times \P_1(\R^n) \to \R$ such that
\begin{equation}
\label{contxF}
 x\mapsto F(x,m) \quad\text{is continuous for any } m\in \P_1(\R^n) \,,
 \end{equation}
  and $m \mapsto F(x,m)$ is continuous with respect to the Kantorovich-Rubinstein distance $d_1$ in $ \P_1(\R^n)$, locally uniformly in $x$, i.e.
\begin{equation}
\label{contF}
\forall \, K\subseteq\R^n \,\text{compact,} \; \sup_{x\in K}|F(x, m) - F(x, \bar m)|\to 0\;\text{ as } \; d_1(m,\bar m)\to 0 .
\end{equation}
We assume the following behavior of $F$ at infinity (cfr. (F4) in \cite{cannarsa2020long}, (H) in \cite{bardi2023eikonal}):
\begin{multline}
\label{Finfinity}
\exists \, K_o\subseteq \R^n \text{ compact, } \d_o>0, \text{ such that, } \forall \, m \in  \P_1(\R^n), \\
\quad \quad \quad \quad \quad \quad \quad \quad \inf_{x\notin K_o}F(x,m) - \min_{x\in K_o}F(x,m)\geq \d_o .\hfill
\end{multline}
Therefore, for any fixed $m\in\mathcal{P}_{1}(\mathds{R}^{n})$, $\inf_{\mathds{R}^{n}} \, F(\cdot,m) = \min _{K_o} \, F(\cdot,m).$

The static Mean Field Game associated to $F$ is the following problem
\begin{equation}
    \label{supp pb}
    \text{find } \bar{m}\in\mathcal{P}_{1}(\mathds{R}^{n}) \text{ s.t.: } \; \text{supp}(\bar{m}) \subseteq \argmin\limits_{x\in\mathds{R}^{n}}\, F(x,\bar{m}),
\end{equation}
where the closed set 
$$\text{supp}(m):=\{z\in \mathds{R}^{n} \,:\, m(U)>0 \, \text{ for each neighborhood } U \text{ of } z\}$$
 is the support of $m$. It is well known from \cite{lions2008} and \cite{cardaliaguetnotes} that problem 
  \eqref{supp pb} is equivalent to the static MFG that consists of finding $\bar{m}\in\mathcal{P}_{1}(\mathds{R}^{n})$ such that
\begin{equation}
    \label{supp pb 2}
    \int_{\mathds{R}^{n}}F(x,\bar{m})\,\dd \bar{m}(x) \; = \inf\limits_{m\in\mathcal{P}_{1}(\mathds{R}^{n})}\int_{\mathds{R}^{n}}F(x,\bar{m})\,\dd m(x).
\end{equation}
Another equivalent formulation, similar to Cournot-Nash equilibria in games with a continuum of players (as, e.g.,  in  \cite{blanchet2014remarks}),  is the following
\[
\text{find } \bar{m}\in\mathcal{P}_{1}(\mathds{R}^{n}) \text{ s.t.: } \; \bar m\left(\argmin F(\cdot, \bar m)\right)= 1 \,.
\]
\begin{remark}\upshape{
The game-theoretic interpretation is as follows. Each player wants to minimize her/his cost $F(\cdot, m)$ associated to the distribution $m$ of the population of players. An equilibrium is a distribution $\bar m$ such that $\bar m$-almost-all players do minimize the cost  $F(\cdot, \bar m)$ associated to the measure $\bar m$ itself. It was shown by Lions  \cite{lions2008, cardaliaguetnotes} that such equilibria arise as limits, for $N\to\infty$, of Nash equilibria in $N$-player games where the generic $i$-th player chooses $x_i$  to  minimize $F(x_{i}, m^{N-1}(x_{-i}))$, 
$m^{N-1}(x_{-i}):= \frac{1}{N-1}\sum_{j\neq i} \delta_{x_{j}}$ being the empirical distribution of the players different form $x_i$.}
\end{remark}

\begin{proposition}
\label{static}
Under the assumptions \eqref{contxF},   \eqref{contF}, and  \eqref{Finfinity}, there exists a solution to \eqref{supp pb}.
\end{proposition}

\begin{proof}
By the assumption \eqref{Finfinity}  the static Mean Field Game \eqref{supp pb} is equivalent to
\begin{equation}
    \label{supp pb proof}
    \text{find } \bar{m}\in\mathcal{P}_{1}(K_{\circ}) \text{ s.t.: } \; \text{supp}(\bar{m}) \subseteq \argmin\limits_{x\in K_{\circ}}\, F(x,\bar{m}).
\end{equation}
Then it is not difficult to check that the latter problem is also equivalent to finding a fixed-point $\bar{m}$ for the set-valued map
\begin{equation}
    \label{setvalued map}
    \mathcal{F}(\bar{m}) = \left\{ m\in\mathcal{P}_{1}(K_{\circ})\, : \, \int\limits_{K_{\circ}} \left[ F(x,\bar{m}) - \inf\limits_{z\in K_{\circ}}F(z,\bar{m})\right]\,\dd m(x) = 0  \right\}.
\end{equation}
Therefore, to prove existence of a solution to \eqref{supp pb}, we will show that the set-valued map $\mathcal{F}(\cdot)$ admits a fixed-point $\bar{m}$. This shall be a consequence of the Kakutani-Glicksberg-Fan theorem \cite{glicksberg1952further, fan1952fixed}. Indeed, since $K_{\circ}$ is a compact, the set $\mathcal{P}(K_{\circ})$ is convex and weakly-$*$ compact. Then, we need to check that the set-valued map $\mathcal{F}(\cdot)$ is upper semi-continuous or, equivalently, that its graph is weakly-$*$ closed \cite[Corollary 1, p.42]{aubin1984differential}. Since the weak-$*$ topology is metrizable, we can check closedness using a sequential argument. Let $(m_{n},\mu_{n})\in \mathcal{P}(K_{\circ})\times \mathcal{F}(m_{n})$ be sequence of measures satisfying 
\begin{equation*}
    m_{n} 	\xrightharpoonup[]{*} m \; \text{ and } \; \mu_{n} \xrightharpoonup[]{*} \mu \quad \text{ in } \mathcal{P}(K_{\circ}).
\end{equation*}
Using the locally uniform continuity of the map  $m\mapsto F(x,m)$, assumption \eqref{contF}, we have $\; F(x,m_{n}) \to F(x,m), \text{ uniformly in  } \, K_{\circ} \,.$ 
        Then $\; \min\limits_{z \in K_{\circ}} F(z,m_{n}) \to \min\limits_{z \in K_{\circ}} F(z,m).  \,$ 
Therefore we can  pass to the limit in
\begin{equation*}
    \int\limits_{K_{\circ}} \left[ F(x,m_{n}) - \inf\limits_{z\in K_{\circ}}F(z,m_{n})\right]\,\dd \mu_{n}(x) = 0
\end{equation*}
which finally yields $\mu\in \mathcal{F}(m)$ and concludes the proof. 
\end{proof}

\begin{remark}\upshape{
Proposition \ref{static}  can also be recovered as a particular case of Theorem 2.2 of  \cite{blanchet2014remarks} about  the Cournot-Nash equilibria for  games with a continuum of players,  occurring when the ``type space" is a singleton. }
\end{remark}

\section{Ergodic MFG}
\label{sec:ergodic}
The first order ergodic Mean Field Game system associated to $F$ and the Lagrangian $L(\dot y, y)=|\dot y|^2/2$ is
\begin{equation}
\label{erg}
\left\{\quad
\begin{aligned}
    c+ \frac 12 |\nabla v(x)|^2 &= F(x,m) & \quad \text{in }\mathds{R}^{n} ,    \\
    -\text{div}(m \nabla v) &= 0 & \quad \text{in }\mathds{R}^{n} ,\\ 
	\int_{\R^n} \dd m(x) &= 1 .
\end{aligned}\right.
\end{equation}
As in \cite{cardaliaguetnotes} a solution of \eqref{erg} is the following.
\begin{definition} 
A solution of \eqref{erg} is a triple $(c, v, m)\in \R\times C(\R^n)\times\P_1(\R^n)$ such that $v$ is a Lipschitz viscosity solution of the first equation, $\nabla v(x)$ exists for $m$ - a.e. $x\in\R^n$, and the second equation is satisfied in the sense of distributions, i.e.,
\begin{equation}
\label{distsol}
\int_{\R^n} \nabla\phi(x)\cdot \nabla v(x) \dd m(x) =0 \quad \forall\, \phi\in C^\infty_c(\R^n) .
\end{equation}
\end{definition}

\begin{remark}\upshape{
The definition in \cite{cannarsa2020long} is more stringent, as it requires $m$ to be a projected Mather measure, i.e., $m(\dd x)$ is the (first) marginal on $\mathds{R}^{n}$ of a Mather measure. For the simple quadratic Hamiltonian considered here this additional condition is automatically satisfied: see the Appendix. }
\end{remark}

\begin{remark}
\label{game2}\upshape{
The game-theoretic interpretation of \eqref{erg} is the following. Given a distribution $m$ of the population of players, the generic agent seeks to  minimize the cost functional of ergodic control 
    \begin{equation*}
        J(x,\alpha,m) = \liminf\limits_{T\to \infty} \frac{1}{T}\int_{0}^{T} \left(\frac{1}{2}|\alpha(s)|^{2} + F(x(s),m)\right )\,\dd s
    \end{equation*}
subject to the  dynamics $\dot{x}(s) = \alpha(s)$, $s \geq 0$ with $x(0)=x$, where the control $\alpha(\cdot)$  is measurable  with values in $\mathds{R}^{n}$. Dynamic Programming leads to the 1st equation  in \eqref{erg}, where $c$ is the (constant)  value function and $-\nabla v(x)$ is the optimal feedback, at least for $v$ regular enough. An equilibrium $m$ of the ergodic MFG is a measure $m$ which is invariant for the optimal flow $\dot{x}(s) = -\nabla v(x(s))$ generated by itself, a fact expressed by the 2nd equation  in \eqref{erg}. Also these games can be associated to the large population limit $N\to\infty$ of $N-$person ergodic differential games, the rigorous connection being proved for control systems affected by a non-degenerate noise and with compact state space \cite{lasry2007mean}, see also \cite{feleqi2013derivation, bardi2016nonlinear, mendico2023differential}.}
\end{remark}

The next result says that any measure $m$ solving the static MFG can be used to build a solution of the ergodic MFG \eqref{erg}  with critical value $c=\min F(\cdot,m)$, and the function $v$ can be uniquely characterized.
\begin{theorem} 
\label{exist}
Assume $m\in \P_1(\R^n)$ satisfies  
\begin{equation}
\label{m}
\text{\upshape supp}(m) \subseteq \argmin_{x\in \R^n} F(x, m) \,,
\end{equation}
 $F(\cdot,m)$ is bounded, continuous, and satisfies \eqref{Finfinity_0}. Set $c_m:=\min_x F(x,m)$. Then there exists $v\in C(\R^n)$ such that 
$(c_m, v, m)$ is a solution of \eqref{erg}. Moreover $v\geq 0$ in $\R^n$ and null on $\M:=\argmin_x F(x,m)$, and it is the unique viscosity solution bounded from below of the Dirichlet problem
\begin{equation}
\label{Dir}
c_m+ \frac 12 |\nabla v(x)|^2 = F(x,m) \;\text{ in } \R^n\setminus\M , \quad v(x)=0 \;\text{ on } \partial\,\M .
\end{equation}
\end{theorem}

\begin{proof}
Since $F(\cdot,m)$ is continuous, bounded, and attains its minimum on $\M$, Theorem 2.1 in \cite{bardi2023eikonal} states the existence of $v$ Lipschitz solving \eqref{Dir} and with the properties stated above. It remains to prove that  $\nabla v(x)$ exists for $m$ - a.e. $x\in\R^n$ and that $m$ solves the second equation in \eqref{erg}.

As in \cite{bardi2023eikonal}, we define $\ell(x):=\sqrt{2(F(x,m)-c_m)}$ and take the square root of \eqref{Dir} to get the equivalent Dirichlet problem for an eikonal equation
\begin{equation}
\label{eik}
 |\nabla v(x)| = \ell(x)  \;\text{ in } \R^n\setminus\M , \quad v(x)=0 \;\text{ on } \partial\,\M .
\end{equation}
Then $v$ can be written as the value function of the control problem
\begin{equation} 
    \label{value}
    v(x)=\inf\limits_{\alpha}\int_{0}^{t_{x}(\alpha)} \ell(y_{x}^{\alpha}(s))\,\text{d}s,
\end{equation}
where $\alpha$  is a measurable function $[0,+\infty) \to B(0,1)$,  the unit ball in $\mathds{R}^{n}$, $t_{x}(\alpha):=\inf\{s \geq 0\,:\, y_{x}^{\alpha}(s) \in \mathfrak{M}\}$, and
\begin{equation*}
    \dot{y}^{\alpha}_{x}(s) = \alpha(s),\,\forall\,s\geq 0,\quad y_{x}^{\alpha}(0)=x ,
\end{equation*}
see, e.g., \cite{bardi2008optimal}.

By the property \eqref{m} we have to consider only $x\in \M$ (and only $x\in \partial\,\M$ is non-trivial). Since $v\geq 0$ and null on $\M$ we have $0\in D^-v(x)$. We claim that $0\in D^+v(x)$. For $h$ small consider $x+h$: if it is in $\M$ then $v(x+h)-v(x)=0$. If $x+h \notin \M$ we plug in \eqref{value} the control 
$\a(s)=-h/|h|$ that leads from $x+h$ to $x$  in time $|h|$, and get
\[
v(x+h)-v(x)\leq |h| \sup_{0\leq s\leq |h|} \ell(x+h - sh/|h|) .
\]
Then, by the continuity of $\ell$,
\[
\limsup_{h\to 0} \frac{v(x+h)-v(x)}{|h|} \leq 0 
\]
which proves the claim. 

To conclude we observe that the integral in \eqref{distsol} can be restricted to $\text{supp}(m) \subseteq \M$, and $\nabla u(x)$ exists and is null for all $x\in \M$.
\end{proof}

\begin{remark}
\label{rem:Mather}\upshape{
Theorem  \ref{exist} can be interpreted in the language of weak KAM theory as follows, using the Appendix: any equilibrium $m$ of the static MFG associated to $F$ is a projected Mather measure for the Hamiltonian $H(p,x)= |p|^2/2 -F(x,m)$ with Ma\~n\'e critical value $-c_m$. }
\end{remark}

By combining Proposition \ref{static} with Theorem \ref{exist} we immediately get the following.
\begin{corollary}
\label{exist2}
Assume $F(\cdot, m)$ is bounded for all $m$ and satisfies \eqref{contxF}, \eqref{contF}, and \eqref{Finfinity}. Then there exists a solution $(c_m, v, m)$  of \eqref{erg}.
\end{corollary}

\begin{remark}
\label{rem:fGJ}\upshape{
An example of $F$ satisfying the assumptions of Corollary \ref{exist2} is 
$$
F(x,m)=f(x)G(x,m) + J(x,m)
$$
 with $f, G, J$ bounded and continuous in $x$, $G$ and $J$ satisfying \eqref{contF} (e.g., convolution operators), $f(x)=0$ for some $x$, and 
\begin{equation*}
\liminf_{|x|\to\infty} f(x)>0 \,,\; G(x,\mu)\geq 1 \text{ and } J(x,\mu)\geq 0 \; \forall \,|x|\geq R \,, \, \mu\in \P_1(\R^n) .
\end{equation*}
Then $c_m\leq 0$ for all $m$. }
\end{remark}

The next result is also an easy consequence of Theorem \ref{exist}. Define 
\begin{equation}
\label{int_argmin}
Int := \bigcap_{m \in  \P_1(\R^n)} \argmin F(\cdot,m)  \,.
 \end{equation}

\begin{corollary}
\label{exist3} Assume $Int \ne\emptyset$, $F$ is bounded 
and satisfies \eqref{contxF} and \eqref{Finfinity}. Then any measure $m$ supported in $Int$ is a solution of  \eqref{m}, and therefore 
there exists $v\in C(\R^n)$ such that 
$(c_m, v, m)$ is a solution of \eqref{erg}. 
\end{corollary}
\begin{remark}
\label{rem:fG}\upshape{
An example of $F$ satisfying the assumptions of Corollary \ref{exist3} is 
$$
F(x,m)=f(x)G(x,m) + g(m)
$$
 with the property that $\liminf_{|x|\to\infty} f(x)>0$ and for some nonempty closed set $\B$
\begin{equation}
\label{fG}
f(x)=0 \;\forall x\in\B , \quad f(x)\geq 0
\text{ and }  G(x,\mu)\geq 1 \;\; \forall \, x\notin \B\,, \, \mu\in \P_1(\R^n) .
\end{equation}
In this case $\B\subseteq Int$,  so any measure $m$ with $\text{supp}(m)\subseteq \B$ satisfies \eqref{m} and $c_m=0$. }
\end{remark}

Another easy consequence of Theorem \ref{exist} is the following generalization of Corollary \ref{exist3}.
\begin{corollary}
\label{exist4}
Assume  $F$ is bounded,
it satisfies \eqref{contxF} and \eqref{Finfinity}, and 
there exists $\B\subseteq \R^n$ such that 
\begin{equation}
\label{argmin2}
\B \ne
 \emptyset, \quad \B\subseteq \bigcap \{  \argmin F(\cdot,m)  : m \in  \P_1(\R^n),  \text{\upshape supp}(m)\subseteq \B \}.
 \end{equation}
 Then any measure $m$ supported in $\B$ is a solution of  \eqref{m} and 
there exists $v\in C(\R^n)$ such that 
$(c_m, v, m)$ is a solution of \eqref{erg}. 
\end{corollary}

\begin{remark}
\label{rem:f+G}\upshape{
An example of $F$ satisfying Corollary \ref{exist4} is the following: 
$$
F(x,m)=f(x)+J(x,m)  + g(m),\; f\geq 0 ,\quad J(x,m)=\int_{\R^n} k(|x-y|) \dd m(y) , \; k\geq 0 ,
$$
where we assume $\liminf_{|x|\to\infty} f(x)>0$ and that for a closed set $\B\ne\emptyset$ and $\d\geq 0$
\begin{equation}
\label{f+G}
f(x)=0 \quad \forall x\in\B , \quad \text{diam}(\B)\leq \d , \quad k(r)=0 \;\forall r\leq\d .
\end{equation}
To prove this claim we first fix $x\in\B$, so that $f(x)=0$,  $|x-y|\leq \d$  for all $y\in\B$, and then $k(|x-y|)=0$ for such $y$.
Next we fix
a measure $m \in  \P_1(\R^n)$ with $\text{supp}(m)\subseteq \B$. Then $J(x,m)=0$ and so $F(x,m)=0 = \min F(\cdot, m)$. \\
If $\d=0$, then $\B$ must be a singleton and the only measure concentrated in $\B$ is a Dirac mass, but we don't know if it is the unique solution of  \eqref{m}.}
\end{remark}

The next result is a converse of Theorem \ref{exist} and says that any solution of the ergodic MFG is associated to a solution of the static MFG.

\begin{proposition}
\label{nec_cond}
If $(c, v, m)$ is a solution of \eqref{erg}, then 
\begin{equation*}
\text{\upshape supp}(m) \subseteq \argmin_{x\in \R^n} F(x, m) 
\end{equation*}
and $c\leq c_m := \min_{x} F(x, m)$.
\end{proposition}

\begin{proof}
From the first equation we get $c\leq \inf_{x} F(x, m)$. Fix $\bar x\in \text{supp}(m)$. Then
\[
\int_{B(\bar x, r)} \dd m(x) =: c_r >0 \quad \forall\, r>0 \,.
\]
 Let $\M := \argmin_{x} F(x, m)$ and assume by contradiction that $\bar x\notin \M$. Then 
\[
F(\bar x, m) - c\geq   F(\bar x, m) -\inf_{x} F(x, m)>0 
\]
and, for some $\d,\eta >0$, $F(x, m) - c\geq \eta >0$ for all $x\in B(\bar x, \d)$. This implies 
\[
|\nabla v(x)|^2 \geq 2\eta \quad \text{for $m$ - a.e. } x\in \text{supp}(m) \cap B(\bar x, \d) ,
\]
and then 
\begin{equation}
\label{1}
\int_{B(\bar x, \d)} |\nabla v(x)|^2 \dd m(x) \geq 2\eta c_\d > 0\,.
\end{equation}
Now we use the second equation in \eqref{erg} by choosing mollifiers $\rho_n$, $k_\eps\in C^\infty_c(\R^n)$, and plugging in \eqref{distsol} the test function
\[
\phi(x):= (\rho_n * v)(x) k_\eps (x) \,, \; 0\leq k_\eps\leq 1, \; k_\eps \equiv 1 \,\text{ in }\, B(\bar x, \d) , \; |\nabla k_\eps|\leq \eps\,.
\]
We compute
\begin{equation*}
 \int_{\R^n} \nabla\phi
\cdot \nabla v
\, \dd m
 =  \int_{\R^n} (\rho_n * v) \nabla k_\eps\cdot  \nabla v
 \, \dd m +    \int_{\R^n} k_\eps \nabla  (\rho_n * v) \cdot  \nabla v
 \, \dd m
=: I_1 + I_2 \,.
\end{equation*}
Since $v$ is Lipschitz, $\|\rho_n * v\|_\infty \leq  \| v\|_\infty$, and $m\in\P_1(\R^n)$, we estimate $I_1$ as follows
\[
|I_1|\leq \int_{\R^n} \eps \| \nabla v\|_\infty (|v(0)|+ \| \nabla v\|_\infty|x|) \dd m(x) \leq C\eps \,.
\]
Next we claim that $I_2 \to \int_{\R^n} k_\eps(x) |\nabla v(x)|^2 \dd m(x)$ as $n\to\infty$. In fact,
\begin{multline*}
\left|I_2 - \int_{\R^n} k_\eps |\nabla v |^2 \dd m\right| = \left| \int_{\R^n} (\rho_n * \nabla v -  \nabla v) \cdot  \nabla v \, k_\eps \,\dd m\right| \leq\\ \|\rho_n * \nabla v -  \nabla v\|_{L^2(m)}  \|  k_\eps |\nabla v|\|_{L^2(m)} 
\end{multline*}
which tends to 0 because $\rho_n * \nabla v \to  \nabla v$ in $L^2(m)$ as $n\to\infty$. Then \eqref{distsol} gives 
\[
0 = O(\eps) + \int_{\R^n} k_\eps |\nabla v|^2 \dd m +o(1)\geq  O(\eps) + \int_{B(\bar x, \d)} |\nabla v
|^2 \dd m +o(1)\,, \; \eps\to 0 , \; n\to\infty ,
\]
which is a contradiction to \eqref{1}. In conclusion, $\bar x\in\M$ and $\inf F(\cdot,m)$ is a min. 
\end{proof}

\begin{remark}
\label{rem:Mather2}\upshape{
In the language of weak KAM theory Proposition \ref{nec_cond} says that any projected Mather measure $m$ of the Hamiltonian $H(p,x)= |p|^2/2 -F(x,m)$ is an equilibrium  of the static MFG associated to $F$. }
\end{remark}

Now we can characterize in a special case all solutions of the ergodic MFG \eqref{erg} by defining
\begin{equation}
\label{argmin}
\A := \bigcup_{m \in  \P_1(\R^n)} \argmin F(\cdot,m) \,,
\end{equation}
and assuming 
\begin{equation}
\label{int=un}
Int = \A \,,
\end{equation}
where $Int$ is defined in \eqref{int_argmin}. Assumption \eqref{int=un} is equivalent to the condition \eqref{strongass} in the Introduction.

\begin{corollary}
\label{uniq}
Under assumption \eqref{int=un} there is a solution $(c_m, v, m)$ to \eqref{erg} if and only if $\text{\upshape supp}(m)\subseteq \A= Int$.
\\
In particular, if $\A=\{ \bar x \}$ is a singleton, then $m=\d_{\bar x}$, the Dirac mass in $\bar x$, is the unique measure in $\P_1(\R^n)$ such that \eqref{erg} has a solution, such  solution is $(c_m, v, m)$ with $c_m=F(\bar x, m)$ and $v$ solving \eqref{Dir}, and any other solution has $c\leq c_m$.
\end{corollary}

\begin{remark}\upshape{
An example of  $F$ such that $\A$ is a singleton is when it can be written as $F(x,m)=f(x)G(x,m)$, as in Remark \ref{rem:fG} with $g\equiv 0$, if $\B$ is a singleton, i.e., 
\[
f(\bar x)=0<f(x),\;\text{ and} \quad G(x,\mu)>0 \quad \forall \, x\ne \bar x\,, \,
\mu\in \P_1(\R^n) .
\]
}
\end{remark}

\begin{remark}\upshape{
Note that the condition for uniqueness in Corolllary \ref{uniq} can be satisfied regardless of any monotonicity of $F$ with respect to $m$, so it has nothing to do with the usual sufficient condition for uniqueness of Lasry and Lions.}
\end{remark}

\section{Approximation by long-time asymptotics}
\label{sec:approx}
In this section we adopt the assumptions of \cite{cannarsa2020long} except the monotonicity of $F$, (F3). Our Lagrangian is the convex conjugate of $H(p)=|p|^2/2$, so it is $L(q)=|q|^2/2$, which is a reversible strict Tonelli Lagrangian (see \cite{cannarsa2020long} for the definitions). So we add the assumptions (F1), i.e., 
\begin{equation}
\label{F1}
\text{$F(\cdot,m)$ bounded in $C^2$ uniformly in $m$,}
\end{equation}
 and (F2),  i.e., 
 \begin{equation}
\label{F2}
\text{$F(x,\cdot)$ Lipschitz in $d_1$ norm, uniformly in $x\in\R^n$.}
 \end{equation}
  About the behavior of $F$ at infinity, we need assumption \eqref{Finfinity} which we recall here
\begin{multline}
\label{Finfinity2}
\exists \, K_o\subseteq \R^n \text{ compact, } \d_o>0, \text{ such that } \forall \, m \in  \P_1(\R^n) 
 \\
\quad \quad \quad \quad \quad \quad \quad \quad \quad 
\inf_{x\notin K_o}F(x,m) - \min_{x\in K_o}F(x,m)\geq \d_o .\hfill
\end{multline}
Under this assumption 
\[
\emptyset \ne \argmin_x F(x,m) \subseteq K_o ,\;  \forall \, m \in  \P_1(\R^n).
\]
We strengthen this property and adopt also (F5) from \cite{cannarsa2020long}:
\begin{equation}
\label{intersec}
Int = \bigcap_{m \in  \P_1(\R^n)} \argmin F(\cdot,m) \ne\emptyset .
\end{equation}
See Remark \ref{rem:fG} for examples that satisfy this condition. Now consider the evolutive MFG with horizon $T>0$ 
\begin{equation}
\label{evMFG}
\left\{
\begin{aligned}
   & -\partial_{t}u^T + \frac{1}{2}|Du^T(x,t)|^{2} = F(x, m^T(t) ), & \quad \text{in }  \mathds{R}^n \times (0, T) ,\\   
    & \partial_{t}m^T  -\text{div}(m^T 
    D u^T) = 0 & \quad \text{in }\mathds{R}^{n}\times (0, T)  ,\\ 
    & u^T(x,T) =\, 0, \quad m^T(0)= m_o .
\end{aligned}\right.
\end{equation}
with $m_o$ satisfying assumption (M) in \cite{cannarsa2020long}, i.e., 
\begin{equation}
\label{ass_m0}
\text{supp}(m_o) \subseteq K_o \text{ and has a density } m_o\in L^\infty(\R^n).
\end{equation}
Here $D:=D_x=\nabla_x$ denotes the gradient with respect to the space variables. Under the assumptions \eqref{F1}, \eqref{F2}, and \eqref{ass_m0}, the system \eqref{evMFG} has a solution, in viscosity sense for the first equation and in the sense of distributions for the second, as in \cite{cannarsa2020long}. Here the solution is not necessarily unique, but  for each $m^T(\cdot)$ continuous in $[0,T]$ there is a unique $u^T$ satisfying the Hamilton-Jacobi equation, given by  
\begin{equation}
\label{optcon}
u^T(x,t)= \min_{y\in AC}  \left\{\int_t^T \left[\frac {|\dot y(s)|^2}2 + F(y(s),m^T(s))\right]\, \dd s , \; y(t)= x \right\} ,
\end{equation}
where $AC$ denotes the space of absolutely continuous curves $[t,T]\to \mathds{R}^{n}$, and the minimization problem on the right hand side of \eqref{optcon} has a solution. Then one can define, as in \cite{cardaliaguetnotes} and \cite{cannarsa2020long}, an optimal flow $\Phi^T : K_o\times [0, T] \to \R^n$ such that $y_x^*(\cdot)=\Phi^T(x,\cdot\,)$ is a solution of \eqref{optcon} with $t=0$ and 
\begin{equation}
\label{measflow}
m^T(s) = \Phi^T(\cdot\, ,s )\sharp m_o ,\quad \forall\, s\in[0,T].
\end{equation}

\begin{remark}
\label{game3}\upshape{
The game-theoretic interpretation is similar to the one for ergodic games of Remark \ref{game2}. To a continuous measure $m^T(\cdot)$, it is associated a value function by \eqref{optcon} and the optimal flow $\Phi^T$, which is driven by the feedback $-D u^T(x, s)$, at least where $u^T$ is regular enough. An equilibrium of the finite-horizon MFG is a flow of measures   $m^T(\cdot)$ pushed forward by the flow $\Phi^T$ generated by itself, a fact expressed by the 2nd equation  in \eqref{evMFG}. In principle there is a connection with the large population limit of $N-$person finite-horizon differential games, which was proved in some special cases, in particular for control systems affected by a non-degenerate noise, with compact state space, and satisfying the Lasry-Lions monotonicity condition, see \cite{cardaliaguet2020remarks, cardaliaguet2019master, carmona2018probabilistic, lacker78mean, arapostathis2017solutions}, the references therein, and \cite{fischer2021asymptotic} for some 1st order MFG.}
\end{remark}

\subsection{Convergence of $m^T$}
In all this section $(u^T, m^T)$ denotes a solution of the system \eqref{evMFG}. The standing assumptions are \eqref{F1}, \eqref{F2}, \eqref{Finfinity2}, \eqref{intersec}, and \eqref{ass_m0}. Notation: for $X\subseteq\R^n$, $B_R(X):=\{ x\in \R^n : \text{dist}(x, X)\leq R \}$. 

The next result is borrowed from \cite{cannarsa2020long}. 

\begin{lemma} 
\label{Lip}
For some $R_o>0$ and all $R>R_o$ there are constants $\chi(R), \chi'(R)$ such that, for any $|x|\leq R$, any $0\leq t\leq T,\, T>1$, any minimizer $y^*$ in \eqref{optcon} satisfies
\begin{equation}
\label{Lip_est}
\sup_{s\in[t, T]} |y^*(s)| \leq \chi ,\quad \sup_{s\in[t, T]} |\dot y^*(s)| \leq \chi ' .
\end{equation}
\end{lemma}

\begin{lemma}
Under the standing assumptions 
\begin{equation}
\label{supp mT}
    \text{\upshape supp}(m^{T}(t)) \subseteq \Phi^{T}(K_{\circ},t),\quad \forall\, t\in [0,T].
\end{equation}
In particular, there is $R_1$ such that
\begin{equation}\label{unifbd supp}
    \text{\upshape supp}(m^{T}(t)) \subseteq B_{R_1}(K_o) ,\quad \forall\, T>1,\, t\in [0,T] .
\end{equation}
\end{lemma}

\begin{proof}
To prove \eqref{supp mT} we use that for any $(x,s)\in K_{\circ}\times [0,T]$ one has $m^{T}(x,s) = \Phi^{T}(\,\cdot\,,s) \sharp m_{\circ}(x)$ and compute, for any
$g\in L^{1}(\mathds{R}^{n})$,
\begin{equation*}
    \begin{aligned}
        \int_{\mathds{R}^{n}} g(x) \, \dd m^{T}(x,s) & = \int_{\mathds{R}^{n}} g(x) \, \dd \Phi^{T}(\,\cdot\,,s) \sharp m_{\circ}(x) = \int_{K_{\circ}} g(\Phi^{T}(x,s)) \, \dd m_{\circ}(x)\\
        & = \int_{\Phi^{T}(K_{\circ},s)} g(x) \, \dd \Phi^{T}(\,\cdot\,,s) \sharp m_{\circ}(x) = \int_{\Phi^{T}(K_{\circ},s)} g(x) \, \dd m^{T}(x,s).
    \end{aligned}
\end{equation*}

On the other hand, if $K_o\subseteq B_{R_o}(0)$, we set $R_1= \chi(R_o)$ and use the first inequality in \eqref{Lip_est} to get 
\begin{equation*}\label{bound}
    \Phi^{T}(K_{\circ},s) \subseteq B_{R_1}(K_o) , 
    \quad \forall \, T>1,\, 0\leq s \leq T
\end{equation*}
and then \eqref{unifbd supp}.
\end{proof}

\begin{lemma}
\label{compactness}
For any $s\in(0,1]$ there is a diverging sequence $T_n$, depending on $s$, and $\bar m(s)\in \mathcal{P}_{1}(\mathds{R}^{n})$,  such that $m^{T_{n}}(sT_{n})$ narrowly converges (i.e. converges weak-$*$) to $\bar{m}(s)$, i.e., 
\[
d_1\big(m^{T_{n}}(sT_{n}), \bar m(s)\big) \to 0 , 
\]
and moreover
\[
\lim _n F(x, m^{T_n}(sT_n)) = F(x, \bar m(s))   
\]
uniformly in $\R^n$.
\end{lemma}

\begin{proof}
Let $T_{n}\to +\infty$ as $n\to +\infty$. A direct consequence of \eqref{unifbd supp} is the tightness of the measures $\{m^{T_{n}}(sT_{n})\}_{n}$ in $\mathcal{P}_{1}(\mathds{R}^{n})$. Applying Prokhorov's compactness Theorem yields existence of a subsequence still denoted by $T_{n}$ such that $m^{T_{n}}(sT_{n})$ narrowly converges to some measure $\bar{m}(s)\in \mathcal{P}_{1}(\mathds{R}^{n})$.

For the second statement we use assumption \eqref{F2} to get, for all $x\in \R^n$, 
  \[
  |F(x, m^{T_n}(sT_n)) - F(x, \bar m(s)) | \leq C d_1\big(m^{T_{n}}(sT_{n}), \bar m(s)\big)  ,
  \]
  where $C$ is the Lipschitz constant of $F$ in the $m$ variable, and the right hand side tends to 0 and is independent of $x$.   
\end{proof}

We recall the definition of $\A$  \eqref{argmin} 
\begin{equation*}
\A := \bigcup_{m \in  \P_1(\R^n)} \argmin F(\cdot,m) ,
\end{equation*} 
and set
\[
d(x):=\text{dist}(x, \A) .
\]
We also define 
\[ 
\bar F(x,m):= F(x,m)-\min_x F(x,m) , 
\]
so $\bar F \geq 0$ and $\geq \d_o$ off $K_o$ for all $m$, by \eqref{Finfinity2}. We will assume the following generalization of (H) in \cite{bardi2023eikonal}:  for all $r>0$, there exists  $\gamma=\gamma(r)>0$ such that
\begin{equation}
\label{H}
        \inf\{\bar F(x,m)\,:\,m\in\P_1(\R^n) ,\, d(x)> r\} \,\geq \, \gamma(r).
\end{equation}
Note that $\gamma$ can be assumed non-decreasing without loss of generality. Moreover, if \eqref{Finfinity} is in force, for all $r$ large enough this inequality holds with $\gamma(r)=\d_o$. 

We can now state the first main result of this section.
\begin{proposition}
 \label{lemma1st_concl}   
Under the standing assumptions and \eqref{H}, for each $s\in(0,1]$ let $\bar{m}(s)$ be any measure weak-$*$ limit of $m^{T_{n}}(sT_{n})$ for some $T_n\to +\infty$, as in Lemma \ref{compactness}. Then 
\begin{equation} 
 \label{1st_concl}   \text{\upshape supp}(\bar{m}(s)) \subseteq \mathcal{A}.
\end{equation}
\end{proposition} 
The proof follows from some auxiliary results. Recall the notation $c_m:=\min\limits_{x} F(x,m)$.  Define the following neighborhood of $\A$
\[
\A_\d:=\bigcup_m \{ x : \bar F(x,m)<\d \} ,
\]
so $\A_\d^c = \{ x : \bar F(x,m)\geq \d \; \forall m \}$. Next define the occupational measure of $\A_\d^c$ by the optimal trajectory $y_x^*$
\[
\rho_x^\d(T) := \frac 1{T} |\{ s\in [0,T] : y^*(s) \notin \A_\d \}| = \frac 1{T} \int_0^T \mathds{1}_{\A_\d^c}(y_x^*(s)) \,\dd s.
\]
where $|I|$ denotes Lebesgue measure of $I\subset \mathds{R}$.

\begin{lemma} 
\label{estim1}
Under the standing assumptions of the section, there is a constant $C>0$ such that
\begin{equation}
\label{est_above}
\rho^\d_x(T) \leq \frac {C\, d(x)}{\d T} .
\end{equation}
\end{lemma}

\begin{proof}
Define $c^T:= \frac 1T\int_0^T c_{m^T(s)} \dd s$. From the definitions of $\rho^\d_x$ and $\A_\d$ we immediately get
\[
\frac 1{T} \int_0^T F(y^*(s),m^T(s))\, \dd s = \frac 1{T} \int_0^T \left(\bar F(y^*(s),m^T(s)) + c_{m^T(s)} \right) \, \dd s \geq \d \rho^\d_x(T) + c^T ,
\]
and so, by \eqref{optcon},
\[
\rho^\d_x(T) \leq \frac {u^T(x,0)- c^T T}{\d T} . 
\]
On the other hand, set $C_F:= \sup_{x, m}\bar F(x,m)$ and choose in \eqref{optcon} the trajectory which first reaches $\A$ in minimal time with speed $|\dot y| =1$, and then stops. Then
\begin{equation}
\label{est_aux}
u^T(x,0)\leq \left(\frac 12 + C_F\right) d(x) + c^TT , 
\end{equation}
and by combining the last two inequalities we reach the conclusion.
\end{proof}

\begin{lemma} 
\label{flow_near_A} 
Under the standing assumptions and \eqref{H}, for any $r>0$ there is $T_o>0$ such that
\begin{equation}
\label{as_stab}
d(\Phi^T(x,sT)) \leq r \,, \quad \forall \, T\geq T_o, \; T_o - \frac{r}{2\chi'}\leq  sT \leq T \,,\; x\in K_o .
\end{equation}
\end{lemma}

\begin{proof}
Let $y^*=y_x^*(\cdot)$ be a solution of the minimization problem \eqref{optcon} with initial time $0$. We first claim that $d(y^*(t))<r$ for some $t\in [0, T]$ if $T$ is large enough. In fact, if $d(y^*(t))\geq r$ for all $t\in [0, \tau]$, then by \eqref{H} 
$$
\bar F(y^*(t), m) \geq \gamma(r) \quad \forall \, t\in [0, \tau]\, , \forall \,m \,,
$$
and so, as in the proof of Lemma \ref{estim1}, 
\[
u^T(x,0)\geq \int_0^T \left(\bar F(y^*(s),m^T(s)) + c_{m^T(s)} \right) \, \dd s \geq \tau\gamma(r) + Tc^T .
\]
Combining this with \eqref{est_aux} we get
\[
\tau \leq \tau_1 :=Cd(x)/ \gamma(r) \,
\]
which gives a contradiction if $t>\tau_1$, and therefore proves the claim if $T> \tau_1$.

Next we want to prove that the inequality $d(y^*(t))<r$ remains true for all $t\in (\tau_2, T]$ for a suitable $\tau_2$. If $\tau_2=\tau_1$ does not do the job,  there is $t_1>\tau_1$ such that $d(y^{*}(t_1))=r$.
By Lemma \ref{Lip},  $|\dot{y}^{*}(\cdot)|\leq \chi'$ and hence $|y^{*}(t)-y^{*}(\tau)|\leq \chi'|t-\tau|$. Then, 
\begin{equation*} 
d(y^{*}(t)) >  r/2 \quad \forall\,t\in \left]t_1-\frac{r}{2\chi'}, \,  t_1+\frac{r}{2\chi'}\right[  =: I ,
\end{equation*}
which implies 
$$
\bar F(y^*(t), m) \geq \gamma(r/2) \quad \forall \, t\in I 
\, , \forall \,m \,.
$$
Then, if $T\geq t_1+\frac{r}{2\chi'}$,
\begin{equation*}
    T\,\rho^{\gamma(r/2)}(T) \geq r/\chi' \,.
\end{equation*}
If $\tau_2=t_1$ does not do the desired job, we iterate the construction and find $t_1<t_2< ...<t_N$ such that $d(y^{*}(t_j))=r$ and $t_{j+1} - t_j\geq r/\chi'$.
By the preceding argument, if $T\geq t_N+\frac{r}{2\chi'}$,
\begin{equation*}
    T\,\rho^{\gamma(r/2)}(T) \geq N r/\chi' \,.
\end{equation*}
Then by Lemma \ref{estim1}
\begin{equation*}
 N\leq  \frac{C\chi' d(x)}{r \gamma(r/2)} \,,
\end{equation*}
and the procedure ends in a finite number of steps with $\tau_2=t_N$. Thus we get the conclusion by choosing $T_o:= t_N+\frac{r}{2\chi'}$.
\end{proof}

\begin{corollary} 
\label{Phiconv}
Under the standing assumptions and \eqref{H}, 
\[
\lim_{T\to +\infty} d(\Phi^T(x,sT)) = 0 \; \text{ and } \;\lim_{T\to +\infty}\sup_{x\in \text{\upshape supp}(m^T(sT))} d(x) =0
\]
locally uniformly in $s\in (0, 1]$ and {uniformly in } $x\in K_o$.
\end{corollary}

\begin{proof}
The first statement follows immediately from \eqref{as_stab}, and the second is got by combining it with \eqref{supp mT}.
\end{proof}

\begin{lemma}
\label{Amb}
Under the standing assumptions, 
for any $s\in (0, 1]$ one has 
\begin{equation*}
    \forall\,x\in \text{\upshape supp}(\bar{m}(s)), \, \exists\, x_{n} \in \text{\upshape supp}(m^{T_{n}}(sT_{n}))\quad \text{s.t. } \; \lim\limits_{n\to +\infty} x_{n} = x,
\end{equation*}
where $T_{n}(s)$ is the diverging sequence found in Lemma \ref{compactness}. 
\end{lemma}
\begin{proof}
This is a direct consequence of Lemma \ref{compactness}  and \cite[Prop. 5.1.8, p.112]{ambrosio2005gradient}.
\end{proof}

\begin{proof}{\em (of Proposition \ref{lemma1st_concl})}.
It's enough to combine  Corollary \ref{Phiconv} with Lemma \ref{Amb}.
\end{proof} 
Proposition \ref{lemma1st_concl} gives a strong information on the weak-$*$ limits  $\bar{m}(s)$ of the sequences $m^{T_{n}}(sT_{n})$, but it is not yet enough to get a solution of the ergodic MFG system \eqref{erg}. For that goal we introduce the additional assumption \eqref{int=un}, i.e., $Int=\A$
, that is,
\begin{equation}
\label{cap=cup}
\argmin_x F(x,m) = \A \quad \forall\, m \in  \P_1(\R^n) \,.
\end{equation}

 \begin{theorem}
\label{conv-m}  
Suppose the standing assumptions, \eqref{H}, and \eqref{cap=cup}  hold true. Then, for any $s\in (0, 1]$, the family $\{ m^{T}(sT) , T>1 \}$  is precompact in $\P_1(\R^n)$, and any weak-$*$ limit $\bar m(s)$ of a sequence $m^{T_{n}}(sT_{n})$ with $T_n\to \infty$ solves the static game \eqref{supp pb}, as well as the ergodic MFG system \eqref{erg} with $c=c_{\bar m(s)}$ and $v$ given by Theorem  \ref{exist}.
 \end{theorem}
 
\begin{proof}
From \eqref{1st_concl} and \eqref{cap=cup} $\bar m$ satisfies the static MFG \eqref{m} and then, by Theorem \ref{exist}, there is $v$ such that $(c_{\bar m}, v, \bar m)$ is a solution of \eqref{Dir} and \eqref{erg}.
\end{proof} 

\begin{remark}
\label{many_s}\upshape{
  Theorem \ref{conv-m}  gives a remarkable property of all limits of sequences $m^{T_{n}}(sT_{n})$ as  $T_n\to \infty$ and with a fixed scaling parameter $s\in (0, 1]$, but it does not allow to identify the limit because the static and ergodic MFG \eqref{supp pb} and \eqref{erg} have in general infinitely many solutions under the current assumptions. In fact any probability measure supported in $\A$ is an equilibrium of the static MFG. In Section \ref{sec:single} we discuss the special case that $\A$ is a singleton: then all $\bar m(s)$ coincide with the Dirac delta concentrated in $\A$ and we show that $m^{T}(sT)$ converges to such measure uniformly in $s$.
 }
 \end{remark}

\begin{remark}
\label{example1}\upshape{
 The example of Remark \ref{rem:fG} satisfies condition \eqref{cap=cup}.
 The assumption \eqref{H} can be re-interpreted under this condition. Since $\bar F(x,m)=0$ for all $m$ if $d(x)=0$, $\gamma$ is a uniform modulus of continuity of $F(\cdot,m)$ for small $r$. }
 \end{remark}
 
\begin{remark}
\label{moregeneral}\upshape{
The same result holds if we define $\A=  \cup_{0\leq s \leq T} \argmin  F(\cdot,m^T(s)),$ and assume that it coincides with $\cap_{0\leq s \leq T} \argmin  F(\cdot,m^T(s)).$ 
Then we can repeat the same proof with now
\[
\A_\d:=\bigcup_{0\leq s \leq T} \{ x : \bar F(x,m^T(s))<\d \} .
\]
These sets are smaller, so in principle the results gets better, but they depend on the initial mass $m_o$ and the solution $m^T$, not just on the data.}
\end{remark}
 
\subsection{Convergence of $u^T$}
We continue the study of  $(u^T, m^T)$  solving the system \eqref{evMFG},  and now look at $u^T$. We will need the following standard estimates, see, e.g., Lemma 2.3 in \cite{bardi2023eikonal}.
\begin{lemma}
\label{estimates} 
Set  $M=\sup_{x,m} |F(x,m)|$.  Under assumption \eqref{F1}, for all $(x,t)\in\R^n\times(0, T)$, $T>0$,
\begin{equation*}
   \inf_{x,m}F(x,m)  \leq \frac {u^T(x, t)}{T-t} \leq \sup_{x,m}F(x,m)    \,,
\end{equation*}
\begin{equation*}
  |\partial_t u^T(x,t)| \leq M \; \text{a.e.}, \quad | D u^T(x,t)| \leq \sqrt{4M} \; \text{a.e.}
 \end{equation*}
\end{lemma}

To go ahead we need to strengthen the assumption \eqref{cap=cup} and assume that also the value $c_m$ of the min  is the same for all $m$, namely,
\begin{equation}
\label{c*}  c_* = \min_x F(x,m)   \quad \forall m.
 \end{equation}
 This is satisfied by the example in Remark \ref{rem:fG} if $g\equiv 0$.
\begin{lemma}
\label{conv-u}
Under the assumptions of Theorem \ref{conv-m} and \eqref{c*}, 
\begin{equation}
\label{lim-uT}
\left| \frac {u^T(x,t)}{T-t}-c_*\right| \leq \frac {\sqrt{4M} d(x)}{T-t} .
 \end{equation}
 In particular
 \[
 \lim_{T\to +\infty}  \frac {u^{T}(x, t)}{T-t} = c_*\quad\text{locally uniformly in } \R^n\times [0,T) \,.
 \]
\end{lemma}

\begin{proof}
Consider first $\bar x\in \A$, so $F(\bar x,m^{T}(s))=c_*$ for all $0\leq s \leq T$. We claim that 
 $$
 u^{T}(\overline{x},t)=(T-t) c_* \quad\text{ for all } T.
 $$
  In fact, for such $\bar x$, by choosing $\dot y_{\cdot}\equiv 0$ in \eqref{optcon}, we get
 \begin{equation*}
 u^{T}(\overline{x},t)   \leq \int_{t}^{T}F(\overline{x},m^{T}(s) )\,\dd s = (T-t) c_* .
\end{equation*}
 The other inequality $\geq$ is true for all $x\in\mathds{R}^{n}$ because 
\[
\int_t^T \left[\frac {|\dot y(s)|^2}2 + F(y(s),m^T(s))\right]\, \dd s \geq (T-t) c_*
\]
for all $x=y(0)$ and $T>0$,  and so $u^T(x,t)\geq (T-t) c_*$ by \eqref{optcon}.

Now  the gradient estimate
$
|D u^T(x,t)| \leq \sqrt{4M} $ a.e. in Lemma \ref{estimates} 
 immediately gives \eqref{lim-uT}.
\end{proof} 

\begin{remark}\upshape{
The same proof holds if  \eqref{c*} is replaced by the weaker assumption
\begin{equation*}
\label{c*bis}  c_* = \min_x F(x,m^T(s))   \quad \forall \, 0\leq s \leq T , T>1 .
 \end{equation*}}
\end{remark}

Now define, as in \cite{cardaliaguet2013long}, 
\[
v^T(x, s):= u^T(x, sT) , \quad \nu^T(s)  := m^T(sT) ,\quad  s\in [0,1] ,\, x\in\R^n .
\]
The main result in \cite{cardaliaguet2013long} and \cite{cannarsa2020long} is about the convergence of $v^T(\cdot, s)$ to $c_*(1-s)$ as $T\to\infty$, where instead of $c_*$ they have the critical value of a more general Hamiltonian. From Lemma \ref{conv-u} we get the following.
\begin{theorem}
\label{thm:conv}
Under the assumptions of Theorem \ref{conv-m} and \eqref{c*}, we have for all $R>0$
\begin{equation}
\label{est-conv}
\sup_{0\leq s \leq 1} \left\| \frac {v^T(\cdot, s)}T - c_*(1-s)\right\|_{L^\infty(B_R)} \leq  \frac {C(R)}T \,.  
 \end{equation}
\end{theorem}

\begin{proof}
The left hand side is $0$ for $s=1$. On the other hand, for $0\leq s <1$ and all $x$, the estimate  \eqref{lim-uT} gives,
\[
| v^T(x, s) - c_*T(1-s) |\leq  \sqrt{4M} d(x) \,,
\]
which implies \eqref{est-conv}, for some constant $C(R)\geq \sqrt{4M} d(x)$ for all $|x|\leq R$.
\end{proof}

\begin{remark}\upshape{
Note that \eqref{est-conv} is better than the estimates in \cite{cardaliaguet2013long} and \cite{cannarsa2020long}, where the right hand side is  $\leq C/\sqrt T$  in the best case. }
\end{remark}

\begin{remark} \upshape{
Examples of cost $F$ satisfying all the assumptions of the Theorems \ref{conv-m} and \ref{thm:conv} are of the form
\[
F(x,m)=f(x) g\left(\int_{\R^n}k(x,y) \,\dd m(y)\right) \,, \quad f\geq 0\,,\, k\geq 0\,,\, g(r)\geq 1 \;\forall\, r\geq 0 \,,
\]
where $f$ and $g : [0,+\infty)\to \R$ are $C^2$ with bounded first and second derivatives, $k\in C(\R^{2n})$  has first and second derivatives with respect to $x$, bounded uniformly in $y$. Then \eqref{F1} and \eqref{F2} are verified.
Moreover, if $f$ vanishes at some point and $\liminf_{|x|\to\infty} f(x)>0$, then also \eqref{Finfinity2} holds. Finally,  \eqref{cap=cup}  and \eqref{c*} are satisfied with $\A =\argmin f$ and $c_*=0$. }
\end{remark}

We would like to go one step further: in the spirit of weak KAM theory (see, e.g., Theorem 2.2 in \cite{bardi2023eikonal} and the references therein) one might conjecture that $v^T(\cdot, s) - c_*T(1-s)$ converges along subsequences to a solution of the first equation in the ergodic problem \eqref{erg}. In general we can prove a weaker result. Define, for $0<s\leq 1$,
\[
F_*(x,s):= \liminf_{T\to\infty, \tau\to s} F(x, m^T(\tau T)) , \quad F^*(x,s):= \limsup_{T\to\infty, \tau\to s} F(x, m^T(\tau T)) \,,
\]
\[
w^T(x, s):= v^{T}(x, s)  - c_*T(1- s ) \,,
 \]
and the (partial) relaxed semilimits
\[
\underline v(x,s) := \liminf_{T\to\infty, \tau\to s} w^T(x,\tau) \,,
\quad \bar v(x,s) := \limsup_{T\to\infty, \tau\to s}  w^T(x,\tau) \,.
\]

\begin{theorem} 
\label{conv-uT}
Under the assumptions of Theorem \ref{conv-m} and \eqref{c*}, $\underline v(\cdot, s)$ is a supersolution of 
\begin{equation}
\label{eiko1}
c_*+ \frac 12 |\nabla_x v|^2 = F_*(x,s) \;\text{ in } \R^n ,
\end{equation}
and 
 $\bar v(\cdot, s)$ is a subsolution of 
\begin{equation}
\label{eiko2}
c_*+ \frac 12 |\nabla_x v|^2 = F^*(x,s) \;\text{ in } \R^n .
\end{equation}
\end{theorem}

\begin{proof}
Note that $w^T$ solves 
\[
-\frac{\de_s w^T}{T} + c_*+ \frac 12 |\nabla_x w^T|^2 = F(x, m^T(s T)) \,.
\]
We are going to take the usual relaxed semilimits of the theory of viscosity solutions as $T\to\infty$. Note that, by assumption \eqref{F1}, 
\[
F_*(x,s) = \liminf_{y\to x, T\to\infty, \tau\to s} F(y, m^T(\tau T)) , 
\]
and the same holds for $F^*$. Similarly, the gradient estimate in Lemma \ref{estimates} implies 
\[
\underline v(x,s) = \liminf_{y\to x, T\to\infty, \tau\to s} w^T(y,\tau) \,,
\]
and the same holds for $\bar v$. 
Then the classical stability property of viscosity sub- and supersolutions with respect to relaxed semilimits (see, e.g., \cite{bardi2008optimal}) gives the conclusion.
\end{proof} 

\begin{remark}
\label{comments}\upshape{
Consider the measures $\bar m(s)$ found in Lemma \ref{compactness}. They satisfy
\[
F_*(x,s) \leq F(x, \bar m(s)) \leq F^*(x,s) \quad \forall\, x\in\R^n \,.
\]
If, for some $s$, $F_*(x,s) =F^*(x,s)$ for all $x$, then $v(\cdot,s):= \underline v(\cdot, s) = \bar v(\cdot, s)$ solves
\begin{equation}
\label{eiko}
c_*+ \frac 12 |\nabla_x v|^2 = F(x, \bar m(s)) \;\text{ in } \R^n .
\end{equation}
Then, by Theorem \ref{exist}, there exists $v(\cdot,s)$ such that $(c_*, v(\cdot,s), \bar m(s))$ solves the ergodic problem \eqref{erg}. 

By assumption \eqref{F2} the limits in the definitions of $F_*$ and $F^*$ are uniform in $x$. Then the equality $F_*(\cdot,s) =F^*(\cdot,s)$ occurs when $F(x, m^T(s T))$ converges as $T\to\infty$ locally uniformly in $s\in (0, 1]$. This is the case in the example of next section. }
\end{remark}

\subsection{Stronger convergence for $\A$  singleton} 
\label{sec:single}
In this section we give a much more precise result on the long-time behavior of the solution $(u^T, m^T)$ of the system \eqref{evMFG} under the additional assumption  that $\A=\{x_*\}$ is a singleton.  In this special case only $\bar m = \d_{x_*}$ satisfies \eqref {1st_concl}. Therefore  $d_1(m^T(sT), \d_{x_*})\to 0$ for all $0<s\leq 1$, by Proposition \ref{lemma1st_concl}. This can be improved as follows.

\begin{lemma} 
\label{lemma-ex}
The limits
\[
\lim_{T\to \infty} d_1(m^T(sT), \d_{x_*}) = 0 \;\text{ and }\;
\lim_{T\to \infty} F(x, m^T(sT)) = F(x,  \d_{x_*})
\]
are locally uniform for $s\in (0,1]$.
\end{lemma}

\begin{proof} 
By the Kantorovich-Rubinstein Theorem
\[
d_1(m^T(sT), \d_{x_*})=\sup \left\{ \int_{\R^n} \phi(x) \,\dd \big(m^T(\cdot\,,sT)-\d_{x_*}(\cdot)\big)(x) : \phi\in\text{Lip}(1) \right\}. 
\]
For $ \phi\in\text{Lip}(1)$ (Lipschitz with constant 1) we have
\begin{multline}
\int_{\R^n} \phi(x) \,\dd \big(m^T(\cdot\,,sT)-\d_{x_*}(\cdot)\big)(x) = \int_{\R^n}
 \left(\phi(x) - \phi(x_*)\right)\,\dd m^T(x,sT) \\
\leq \int_{\text{\upshape supp}(m^T(sT))} |x - x_*|\,\dd m^T(x,sT) \leq \sup_{x\in \text{\upshape supp}(m^T(sT))} d(x) \,.
\end{multline}
Then, by Corollary \ref{Phiconv},
\[
d_1(m^T(sT), \d_{x_*}) \leq \sup_{x\in \text{\upshape supp}(m^T(sT))} d(x) \to 0 \quad \text{as } T\to+\infty 
\]
locally uniformly for $s\in(0,1]$.
\\
The second statement follows from assumption \eqref{F2}.
\end{proof} 

 \begin{theorem}
\label{conv-ex}
Under the assumptions of Theorem \ref{conv-uT} and $\A=\{x_*\}$ 
\[
\lim_{T\to \infty}\left(v^T(x, s) - c_*T(1-s)\right) = v(x) \,,
 \]
 locally uniformly for $x\in \R^n$ and $s\in(0,1)$, and the triple $(c_*, v, \d_{x_*})$ is a solution of the ergodic mean field game system \eqref{erg}.
 Moreover  $v$ is non-negative and the unique  viscosity solution bounded from below of 
 \begin{equation}
\label{eiko-ex-Dir}
c_*+ \frac 12 |\nabla_x v|^2 = F(x, \d_{x_*}) 
\;\text{ in } \R^n \setminus\{x_*\} , \quad v(x_*)=0 \,. 
\end{equation}
 \end{theorem}

\begin{proof} 
By Lemma \ref{lemma-ex}
\[
F_*(x,s)=F^*(x,s)=F(x,  \d_{x_*} ) \quad \forall \, 0<s\leq 1 \,,
\]  
so the PDEs \eqref{eiko1} and \eqref{eiko2} both become
 \begin{equation}
\label{eiko-ex}
c_*+ \frac 12 |\nabla_x v|^2 = F(x, \d_{x_*}) \;\text{ in } \R^n ,
\end{equation}
which does not depend on $s$. 

By Lemma \ref{conv-u} $w^T(x_*, s)=0$ for all $s\in(0,1)$, and then also $\underline v(x_*, s)=0$, $\bar v(x_*, s)=0$. In the proof of Lemma  \ref{conv-u} we also showed that $u^T(x,t)\geq(T-t)c_*$ for all $x$ and $t$, which gives $w^T(x, s)\geq 0$ and therefore  $\underline v(x, s)\geq 0$, $\bar v(x, s)\geq 0$ for all $x$ and $s>0$.
By Theorem \ref{conv-uT} and a standard comparison principle for eikonal-type equations we get $\bar v(x, s)\leq \underline v(x, s)$, and so they coincide. 
This implies the locally uniform convergence of  $w^T(x, s)$ to some $v(x, s)$, which solves  \eqref{eiko-ex}, vanishes for $x=x_*$, and is non-negative. Again by a comparison principle there is at most one such function, and then $v$ is independent of $s$.

By Theorem \ref{conv-m} we conclude that $(c_*, v, \d_{x_*})$ solves \eqref{erg}.
\end{proof} 

\section{Appendix: On Mather measures}
\label{App1}

 We say that $\mu$ is a Mather measure for a Lagrangian $L_{m}$ if it is a probability measure satisfying
    \begin{equation*}
        \int L_{m}(x,\alpha) \mu(\dd x, \dd \alpha) = \inf\limits_{\nu \in \mathcal{M}_{L}} \int L_{m}(x,\alpha) \nu(\dd x, \dd \alpha)
    \end{equation*}
where $\mathcal{M}_{L}$ is the set of invariant probability measures (see, e.g., p. 367 in \cite{cannarsa2020long}) and $L_{m}(x,\alpha) := \frac{1}{2}|\alpha|^{2} + F(x,m)$. Since $L_m$ is reversible, i.e., $L_{m}(x,\alpha) = L_{m}(x, -\alpha)$, from equations (6) and (7)  in \cite{cannarsa2020long} we get 
\begin{equation*}
\inf\limits_{x\in \mathds{R}^{n}} L_{m}(x,0) = \inf\limits_{\nu \in \mathcal{M}_{L}} \int L_{m}(x,\alpha) \nu(\dd x, \dd \alpha).
\end{equation*}
The l.h.s. of the latter equality can be written as 
\begin{equation*}
\inf\limits_{x\in \mathds{R}^{n}} L_{m}(x,0) = \inf\limits_{x\in \mathds{R}^{n}} F(x,m) = \inf\limits_{\eta \in \mathcal{P}_{1}(\mathds{R}^{n})} \int F(x,m)\eta(\dd x)
\end{equation*}
because $\eta \in \mathcal{P}_{1}(\mathds{R}^{n})$ is a minimizer of $\eta \mapsto \int F(x,m)\eta(\dd x)$   if and only if it is supported on the infimum of $x\mapsto F(x,m)$. On the other hand, since $L_{m}$ 
 is separable in $(x,\alpha)$, 
\[
 \inf\limits_{x\in \mathds{R}^{n}} L_{m}(x,0) = \inf\limits_{ \eta \in \mathcal{P}_{1}(\mathds{R}^{n}) } \int L_{m}(x,\alpha) \, \eta(\dd x)\!\otimes \!\delta_{0}(\dd \alpha)\,.
 \]
By combining the last three equalities we get that any solution $\bar{m}$ of the static Mean Field Game  \eqref{supp pb 2} is automatically a projected Mather measure for the Lagrangian $L_{\bar{m}}$, which coincides with statement (iii) in \cite[Definition 7]{cannarsa2020long}, and moreover $\bar{m}(\dd x)\!\otimes \!\delta_{0}(\dd \alpha)$ is a Mather measure for  $L_{\bar{m}}$.

\bibliography{bibliography}
\bibliographystyle{siam}

\end{document}